\documentclass[twoside,10pt]{article}

\usepackage[margin=0.5in]{geometry}

\usepackage{fullpage}
\usepackage{enumerate}
\usepackage{authblk}
\usepackage[colorinlistoftodos]{todonotes}
\usepackage{verbatim}
\usepackage{section, amsthm, textcase, setspace, amssymb, lineno, amsmath, amssymb, amsfonts, latexsym, fancyhdr, longtable}
\usepackage{epsfig, graphicx, pstricks,pst-grad,pst-text,tikz, tkz-berge,colortbl}
\usepackage{epsf}
\usepackage{graphicx, color}
\usepackage{float}
\usepackage[rflt]{floatflt}
\usepackage{amsfonts}
\usepackage{latexsym}
\usepackage{xcolor}
\usepackage{mathbbol}
\usepackage{bm}
\usepackage[titletoc,toc,title]{appendix}
\usepackage{titlesec}

\definecolor{darkgreen}{rgb}{0, 0.5, 0}

\newtheorem{theorem}{Theorem}
\newtheorem{lemma}{Lemma}
\newtheorem{corollary}{Corollary}
\newtheorem{obs}{Observation}
\newtheorem{conj}{Conjecture}
\newtheorem{definition}{Definition}
\newtheorem{ex}{Example}
\newtheorem{case}{Case}

\newcommand{\mf}{\mathfrak}

\newcommand{\A}{{\rm A}}

\newcommand{\dd}{\hspace{.1cm}|\hspace{.1cm}}

\newcommand{\Z}{\mathbb Z}
\newcommand{\C}{\mathbb C}

\newcommand{\ind}{{\rm ind \hspace{.1cm}}}

\title{On compositions associated to seaweed subalgebras of $\mathfrak{sl}(n)$ }

\author[*]{Vincent E. Coll, Jr.}
\author[*]{Aria Dougherty}
\author[**]{Matthew Hyatt}
\author[ ]{Andrew W. Mayers}
\author[*]{Nick W. Mayers}

\affil[*]{Department of Mathematics, Lehigh University, Bethlehem, PA, 18015}

\affil[**]{Department of Mathematics, Pace University, NY, USA}

\begin{document}
\maketitle
\noindent
\begin{abstract}
\noindent
A standard seaweed subalgebra of $A_{n-1}=\mathfrak{sl}(n)$ may be parametrized by a pair of compositions of the positive integer $n$.  For all $n$ and certain $k(n)$, we provide 
closed-form formulas and the generating functions for $C(n,k)$ -- the number of parametrizing pairs which yield a seaweed subalgebra of $\mathfrak{sl}(n)$ of index $k$. Our analysis sets the framework for addressing similar questions in the other classical families.

\end{abstract}



\tableofcontents
\section{Introduction}

The index of a Lie algebra $\mathfrak{g}$ is an important algebraic invariant and is bounded by the algebra's rank \textbf{\cite{Panyushev1}}:  ind $\mathfrak{g}\leq$ rk $\mathfrak{g}$, with equality when $\mathfrak{g}$ is reductive. More formally, let $\mathfrak{g}$ be a finite dimensional Lie algebra over 
$\C$.  The index of $\mf{g}$ is given by 
  
\[\ind \mf{g}=\min_{f\in \mf{g^*}} \dim  (\ker (B_f)),\]

\noindent where $f$ is a linear form on $\mf{g}$ and $B_f$ is the associated skew-symmetric \textit{Kirillov form} defined by $B_f(x,y)=f([x,y])$ for all $x,y\in\mf{g}$.

Seaweed algebras were first introduced by Dergachev and A. Kirillov in \textbf{\cite{dk}}, where the impetus for their study was to find a setting in which the computation of a Lie algebra's index might be carried out with relative ease.
To this end, they consider certain subalgebras of $\mathfrak{sl}(n)$ which are parametrized by a pair of compositions of $n$.  Recall that a composition of $n$ is a sequence $\textbf{\textit{a}}=(a_1,\dots,a_m)$ where each $a_i$ is a positive integer and $\sum a_i=n$.  If $V$ is an $n$-dimensional vector space with a basis 
$\mathcal{E}=\{e_1,\dots, e_n \}$, let $\textbf{\textit{a}}=(a_1,\dots,a_m)$ and $\textbf{\textit{b}}=(b_1,\dots,b_l)$ be two compositions of $n$ and consider the flags 

$$
\{0\} \subset V_1 \subset \cdots \subset V_{m-1} \subset V_m =V~~~\text{and}~~~ V=W_0\supset W_1\supset \cdots \supset W_t=\{0\}, 
$$
where $V_i=\text{span}\{e_1,\dots, e_{a_1+\cdots +a_i}\}$ and $W_j=\text{span}\{e_{b_1+\cdots +b_j+1},\dots, e_n\}$.  The subalgebra of $A_{n-1}=\mathfrak{sl}(n)$ preserving these flags is called a \textit{seaweed Lie algebra}, or simply \textit{seaweed}, and is denoted by the symbol(s) $\mathfrak{p}^\A_n(\textbf{\textit{a}},\textbf{\textit{b}})=$ 
$\displaystyle \frac{a_1|\cdots|a_m}{b_1|\cdots|b_t}$, which we interchangeably refer to as the \textit{type} of the seaweed. When one of the compositions is trivial, the seaweed is called \textit{parabolic}, and is called \textit{maximal parabolic }if the remaining composition consists of the sum of two terms.  

A basis-free definition is available, but is not necessary for the present discussion.  The evocative ``seaweed'' is descriptive of the shape of the algebra when exhibited in matrix form.   For example, the seaweed algebra $\frac{2|4}{1|2|3}$ consists of traceless matrices of the form depicted on the left side of Figure \ref{fig:seaweed}, where * indicates the possible non-zero entries from the ground field, which we tacitly assume is an algebraically closed field of characteristic zero.

\begin{figure}[H]
$$\begin{tikzpicture}[scale=0.75]
\draw (0,0) -- (0,6);
\draw (0,6) -- (6,6);
\draw (6,6) -- (6,0);
\draw (6,0) -- (0,0);
\draw [line width=3](0,6) -- (0,4);
\draw [line width=3](0,4) -- (2,4);
\draw [line width=3](2,4) -- (2,0);
\draw [line width=3](2,0) -- (6,0);

\draw [line width=3](0,6) -- (1,6);
\draw [line width=3](1,6) -- (1,5);
\draw [line width=3](1,5) -- (3,5);
\draw [line width=3](3,5) -- (3,3);
\draw [line width=3](3,3) -- (6,3);
\draw [line width=3](6,3) -- (6,0);

\draw [dotted] (0,6) -- (6,0);

\node at (.5,5.4) {{\LARGE *}};
\node at (.5,4.4) {{\LARGE *}};
\node at (1.5,4.4) {{\LARGE *}};
\node at (2.5,4.4) {{\LARGE *}};
\node at (2.5,3.4) {{\LARGE *}};
\node at (2.5,2.4) {{\LARGE *}};
\node at (2.5,1.4) {{\LARGE *}};
\node at (2.5,0.4) {{\LARGE *}};
\node at (3.5,2.4) {{\LARGE *}};
\node at (3.5,1.4) {{\LARGE *}};
\node at (3.5,0.4) {{\LARGE *}};
\node at (4.5,2.4) {{\LARGE *}};
\node at (4.5,1.4) {{\LARGE *}};
\node at (5.5,2.4) {{\LARGE *}};
\node at (5.5,1.4) {{\LARGE *}};
\node at (4.5,0.4) {{\LARGE *}};
\node at (5.5,0.4) {{\LARGE *}};

\node at (.5,6.4) {1};
\node at (2,5.4) {2};
\node at (4.5,3.4) {3};
\node at (-0.5,4.9) {2};
\node at (1.5,1.9) {4};

\end{tikzpicture}\hspace{1.5cm}\begin{tikzpicture}[scale=1.3]
	\def\Node{\node [circle,  fill, inner sep=2pt]}
	\node at (0,0) {};
    \Node[label=left:$v_1$] (1) at (0,1.8) {};
	\Node[label=left:$v_2$] (2) at (1,1.8) {};
	\Node[label=left:$v_3$] (3) at (2,1.8) {};
	\Node[label=left:$v_4$] (4) at (3,1.8) {};
	\Node[label=left:$v_5$] (5) at (4,1.8) {};
	\Node[label=left:$v_6$] (6) at (5,1.8) {};
	\draw (1) to[bend left=50] (2);
	\draw (3) to[bend left=50] (6);
	\draw (4) to[bend left=50] (5);
	\draw (2) to[bend right=50] (3);
	\draw (4) to[bend right=50] (6);
\end{tikzpicture}$$
\caption{A seaweed of type $\frac{2|4}{1|2|3}$ and its associated meander}
\label{fig:seaweed}
\end{figure}
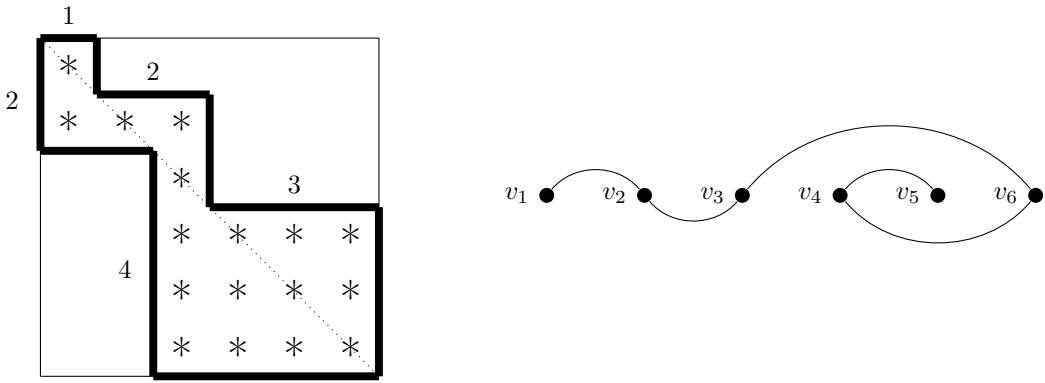

\noindent

Continuing with \textbf{\cite{dk}}, the authors introduce a graph-theoretic representation of the seaweed, called a \textit{meander} and establish that the seaweed's index can be computed by counting the number and types of the connected components of the associated meander -- a fact we will make heavy use of in our study below.  Subsequently, Coll et al \textbf{\cite{Collar}} defined five 
index-preserving, graph-theoretic moves which can be deterministically and iteratively applied to any meander, allowing the meander to be ``wound-down'' to its simplicial homotopy type - a conjugation invariant more granular than the index (see \textbf{\cite{Coll2}}).  These moves can be reversed and applied to the available homotopy types to build up any meander, and so a seaweed, thereby providing an algorithm to construct a seaweed algebra of any rank $n-1$ and index $k$. This allows us to enumerate the number of pairs of compositions of $n$ to yield $C(n,k)$ --  the number of parametrizing pairs which correspond to a seaweed subalgebra of $\mathfrak{sl}(n)$ of index $k$.  An examination of the $C(n,k)$'s suggests certain recursive relationships which we prove and then use to establish closed-form formulas, and attendant generating functions, for $C(n,k)$ when $k$ is equal to $n-1$, $n-2$, or $n-3$. We also provide unrestricted formulas for the number, up to conjugation, of maximal parabolic subalgebras of $\mf{sl}(n)$ with index $k$, as well as providing $C(n,k)$ for a seaweed of type $\frac{a_1|a_2}{b_1|b_2}$.  These latter results follow from the only two available linear greatest common divisor formulas for a seaweed's index based on the parts of its parametrizing pairs.

Our study is inspired by recent work of Duflo who, after the fashion of Coll et al \textbf{\cite{Collar}}, uses certain index-preserving operators on the set of compositions corresponding to a Frobenius 
(index zero) seaweed subalgebra of $\mathfrak{sl}(n)$ to show that if $t$ is the number of parts in the defining compositions, then the 
number of such compositions is a rational polynomial of degree $\left[\frac{t}{2}\right]$ evaluated at $n$. See \textbf{\cite{df}}, Theorem 1.1 (b). 

The structure of the paper is as follows.  In Section \ref{sec: prelim}, we 
develop the formal constructions which allow for the development of the 
$C(n,k)$ table -- which is noted at the beginning of Section \ref{sec: formulas}.  Closed-form formulas for $C(n, n-1)$, $C(n, n-2)$, and $C(n, n-3)$, along with attendant generating functions, are developed in Sections \ref{subsec:n-1}, \ref{subsec:n-2}, and \ref{subsec:n-3}, respectively. The maximal parabolic case
and the development of $C(n,k)$ for the seaweed of type $\frac{a_1|a_2}{b_1|b_2}$ 
are treated separately in Section \ref{subsec:maximal} and 4.3, respectively.

This initial study sets the framework for similar investigations in the other classical families of Lie algebras, where necessary``meandric technologies" have recently become available.

\section{Preliminaries}\label{sec: prelim}
\subsection{Meanders}

In \textbf{\cite{dk}}, Dergachev and A. Kirillov showed how to 
associate to a seaweed $\mathfrak{p}^\A_n(\textbf{\textit{a}},\textbf{\textit{b}})$ a planar graph called a \textit{meander}, which we denote $M^\A_n(\textbf{\textit{a}} \dd \textbf{\textit{b}})$ and construct as follows.  First, label the $n$ vertices of $M^\A_n(\textbf{\textit{a}} \dd \textbf{\textit{b}} )$ as $v_1, v_2, \dots , v_n$ from left to right along a
horizontal line. We then place edges above the horizontal line, called top edges, according to $\textbf{\textit{a}}$ as follows. 
Let $\bm{a}=(a_1,a_2,\dots , a_m)$, and partition the set of vertices into a set partition
by grouping together the first $a_1$ vertices, then the next $a_2$ vertices, and so on, lastly grouping together the final $a_m$ vertices. We call each set within a set partition a \textit{block}. For each block in the set partition determined by $\bm{a}$, add an edge from the first vertex of the block to the last vertex of the block, then add an edge between the second vertex of the block and the second to last vertex of the block, and so on within each block. More explicitly, given vertices $v_j,v_k$ in a block of 
size $a_i$, there is an edge between them if and only if 
$j+k=2(a_1+a_2+\dots+a_{i-1})+a_i+1$.
In the same way, place bottom edges below the horizontal line of vertices according to the blocks in the partition determined by $\bm{b}$. See the right side of Figure \ref{fig:seaweed}.

Every meander consists of a disjoint union of cycles, paths, and points (degenerate paths).  The main result of \textbf{\cite{dk}} is that the index of a seaweed can be computed by counting the number and type of these components in its associated meander. 

\begin{theorem}\label{thm:dk}  \rm{(Dergachev and A. Kirillov, \textbf{\cite{dk}})} 
If $\mf{p}$ is a seaweed subalgebra of $\mf{sl}(n)$, then
$$\ind \mf{p} =2C + P -1,$$
where $C$ is the number of cycles and $P$ is the number of paths in the associated meander.
\end{theorem}

\noindent
\textbf{Example: }  See Figure 1, where $C=0$ and $P=1$.  Hence, the seaweed in this Figure has index 0, so is Frobenius.

\subsection{Homotopy Type}

\begin{definition}
We say that a planar graph has \textit{homotopy type} $H(a_1,a_2,\dots ,a_m)$
if its homotopy type is equivalent to the meander of type 
$\dfrac{a_1|a_2|\dots |a_m}{a_1|a_2|\dots |a_m}$.
That is, a union of $m$ non-concentric subgraphs, where each subgraph
has homotopy type $\frac{a_i}{2}$ concentric circles if $a_i$ is even, and 
$\lfloor a_i/2\rfloor$ concentric circles with a point in the center 
if $a_i$ is odd.
\end{definition}

\noindent
\textbf{Example:} A planar graph with homotopy type
$H(1,5,2)$ is homotopically equivalent to the graph in the following Figure \ref{fig:hom}.

\begin{figure}[H]
\[\begin{tikzpicture}
\def\Node{\node [circle, fill, inner sep=2pt]}
\Node at (0,0){};
\draw (0,0) circle (1.5cm);
\draw (0,0) circle (.75cm);

\Node at (-2.875,0){};

\draw (3.125,0) circle (.75cm);

\end{tikzpicture}\]
\caption{A planar graph with homotopy type $H(1,5,2)$}
\label{fig:hom}
\end{figure}
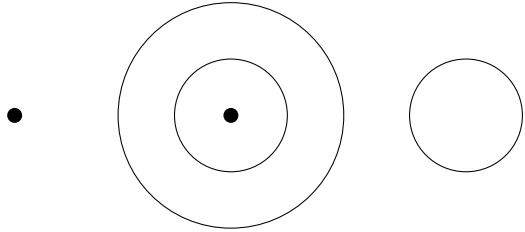

We define the \textit{homotopy type of a seaweed} to be the homotopy type of its corresponding meander.  Unlike the index, the homotopy type of a Lie algebra $\mathfrak{g}$ is not defined directly in terms of $\mathfrak{g}$'s Lie structure.  It is therefore not \textit{a priori} clear to what extent (if at all)  the homotopy type is related to the algebraic structure of the Lie algebra. In fact, the homotopy type is not an algebraic invariant, but it is a conjugation invariant in the sense of the following theorem -- which follows from Theorem 5.3 in the recent paper by Moreau and Yakimova \textbf{\cite{Mor}}.  See also, \textbf{\cite{Coll3}}.

\begin{theorem}\label{homotopy type}
Conjugate seaweed subalgebras of $\mathfrak{sl}(n)$ have the same homotopy type.
\end{theorem}

\noindent
\textit{Remark 1:  } It follows from Theorem \ref{homotopy type} that the homotopy type is a more granular invariant than the index -- and can sometimes be used to show that two seaweeds are not conjugate.  For example, the seaweeds
$\frac{5|3}{3|3|2}$ and $\frac{4|4}{2|4|2}$ have the same dimension (27), rank (7), and index (1), but have homotopy types $H(1,1)$ and $H(2)$, respectively.  So, are not conjugate.

\section{$C(n,k)$ - Formulas and Generating Functions}\label{sec: formulas}
\bigskip
\noindent
Based on the theory above, we first describe an algorithm  to compute $C(n,k)$ -- the number of parametrizing pairs which correspond to a seaweed subalgebra of $\mathfrak{sl}(n)$ of index $k$. We also provide closed-form formulas and attendant generating functions for $C(n,k)$, $k=n-1, n-2$, and $n-3$.  These are addressed in subsections \ref{subsec:n-1}, \ref{subsec:n-2}, and \ref{subsec:n-3}, respectively.  

\begin{center}
\textbf{Algorithm }
\end{center}

Let $n$ be given.

\begin{itemize}
\item List all pairs of compositions summing to $n$,
\item For each pair in this list, construct a meander $M$,
\item Use the signature moves of Appendix \ref{sec:app signature} to wind $M$ down to the simplicial graph which is used to define the homotopy type of $M$,
\item Use Theorem \ref{thm:dk} to compute the index $k$, of $M$. 
\end{itemize}

We obtain the following table. 

\begin{table}[H]
\centering
\begin{tabular}{|c|c|c|c|c|c|c|c|c|c|c|}
\hline
\textbf{$n$\textbackslash $k$} & \textbf{0}    & \textbf{1}     & \textbf{2}     & \textbf{3}     & \textbf{4}     & \textbf{5}     & \textbf{6}     & \textbf{7}    & \textbf{8}    & \textbf{9}   \\ \hline
\textbf{1}        & \cellcolor{blue!25}1    & 0     & 0     & 0     & 0     & 0     & 0     & 0    & 0    & 0   \\ \hline
\textbf{2}        & \cellcolor{red!25}2    & \cellcolor{blue!25}2     & 0     & 0     & 0     & 0     & 0     & 0    & 0    & 0   \\ \hline
\textbf{3}        & \cellcolor{yellow!25}6    & \cellcolor{red!25}6     & \cellcolor{blue!25}4     & 0     & 0     & 0     & 0     & 0    & 0    & 0   \\ \hline
\textbf{4}        & 14   & \cellcolor{yellow!25}26    & \cellcolor{red!25}16    & \cellcolor{blue!25}8     & 0     & 0     & 0     & 0    & 0    & 0   \\ \hline
\textbf{5}        & 34   & 86    & \cellcolor{yellow!25}80    & \cellcolor{red!25}40    & \cellcolor{blue!25}16    & 0     & 0     & 0    & 0    & 0   \\ \hline
\textbf{6}        & 68   & 272   & 330   & \cellcolor{yellow!25}226   & \cellcolor{red!25}96    & \cellcolor{blue!25}32    & 0     & 0    & 0    & 0   \\ \hline
\textbf{7}        & 150  & 764   & 1236  & 1058  & \cellcolor{yellow!25}600   & \cellcolor{red!25}224   & \cellcolor{blue!25}64    & 0    & 0    & 0   \\ \hline
\textbf{8}        & 296  & 2060  & 4216  & 4526  & 3118  & \cellcolor{yellow!25}1528  & \cellcolor{red!25}512   & \cellcolor{blue!25}128  & 0    & 0   \\ \hline
\textbf{9}        & 586  & 5248  & 13528 & 17596 & 14720 & 8674  & \cellcolor{yellow!25}3776  & \cellcolor{red!25}1152 & \cellcolor{blue!25}256  & 0   \\ \hline
\textbf{10}       & 1140 & 12876 & 40820 & 64102 & 63380 & 44480 & 23154 & \cellcolor{yellow!25}9120 & \cellcolor{red!25}2560 & \cellcolor{blue!25}512 \\ \hline
\end{tabular}
\caption{$C(n, k)$}
\label{tab:main}
\end{table}

\subsection{$C(n, n-1)$}\label{subsec:n-1}

\noindent
We first draw attention to the blue cells in Table 1. The result of this short section provides a formula for $C(n, n-1)$.
Consider the combinatorial formula for the index of a meander given by Theorem \ref{thm:dk}.
If a meander with $n$ vertices has index $n-1$, then
all cycles in the meander must contain exactly two vertices, and all paths must contain exactly one vertex. This can only happen if 
the top and bottom compositions defining the meander are equal.
Since there are $2^{n-1}$ compositions of $n$, we have $C(n, n-1)=2^{n-1}$. From this, the following generating function is immediate.

\begin{eqnarray}\label{codim 1}
\sum_{n>0}C(n, n-1)x^n=\frac{x}{1-2x}.
\end{eqnarray}

\subsection{$C(n, n-2)$}\label{subsec:n-2}

\noindent This section concerns the red cells in Table \ref{tab:main}. First, we determine a closed-form formula for the given entries and then develop the associated generating function. 

\begin{theorem}\label{2diag}
$C(n, n-2)=n2^{n-2}$.
\end{theorem}
\begin{proof}
Given a meander with $n$ vertices, we  consider all meanders whose index is $n-2$. First, note that the number of vertices in any cycle is always even. If a cycle contains at least 4 vertices, then the index will be at most $n-3$, so this cannot happen. It follows that a meander whose index is $n-2$ must consist of cycles containing exactly two vertices, one path containing exactly two vertices, and all of the remaining paths must contain exactly one vertex. Furthermore, in order to obtain such a configuration, the unique path containing exactly two vertices must be the outer most edge in a block.
Such a meander must have the form 
\begin{equation}\label{eq:n,n-2}
\dfrac{\bm{a}|i|\bm{c}}{\bm{a}|1|i-2|1|\bm{c}}
\hspace{.5cm}\text{or}\hspace{.5cm}
\dfrac{\bm{a}|1|i-2|1|\bm{c}}{\bm{a}|i|\bm{c}},
\end{equation}
for some integer $i$ such that $2\leq i \leq n$, where
$\bm{a}$ and $\bm{c}$ are (possibly empty) compositions.
Here, a block of size $\bm{a}$ denotes a sequence of blocks 
whose sizes are the parts of $\bm{a}$.


If we first suppose that $\bm{a}$ is empty and $\bm{c}$ is nonempty,  
then the parts of $\bm{c}$ must sum to $n-i$, where $2\leq i\leq n-1$.  So, there
are $2\left(2^{n-i-1}\right)=2^{n-i}$ choices for $\bm{c}$, where the extra factor
of 2 takes into account our two choices in \eqref{eq:n,n-2}.
We get an identical contribution to $C(n,n-2)$ if $\bm{c}$ empty and $\bm{a}$ is nonempty,
and there are 2 additional meanders when both $\bm{a}$ and $\bm{c}$ are empty.
Summarizing, if either $\bm{a}$ or $\bm{c}$ or both are empty, we get a contribution of
\[2+2\sum_{i=2}^{n-1}2^{n-i}=2\sum_{i=2}^n 2^{n-i}=2^n-2.\]



Next, suppose that both $\bm{a}$ and $\bm{c}$ are nonempty. Letting $j$ denote
the sum of the parts of $\bm{a}$, there are $2\cdot2^{j-1}\cdot2^{n-i-j-1}=2^{n-i-1}$
choices for the compositions $\bm{a}$ and $\bm{c}$, where again the extra factor
of 2 takes into account our two choices in \eqref{eq:n,n-2}. 
Summing over $i$ and $j$ we get a contribution of
\[\sum_{i=2}^{n-2}\sum_{j=1}^{n-i-1}2^{n-i-1}
=2+2^{n-2}(n-4)=2+n2^{n-2}-2^{n}.\]



Combining both cases above we find that
\[C(n,n-2)
=2^{n}-2+2+n2^{n-2}-2^n=n2^{n-2}.\]

\end{proof}

\noindent 
Using the above formula we are able to determine a generating function for the $C(n, n-2)$.

\begin{corollary}\label{2diagc1}
For $n\ge 2$
\begin{eqnarray}\label{codim 2}
\sum_{n>0}C(n, n-2)x^n=\frac{2x^2-2x^3}{(1-2x)^2}.
\end{eqnarray}
\end{corollary}

\subsection{$C(n, n-3)$}\label{subsec:n-3}

The results of this section concern the yellow cells of Table \ref{tab:main}. First, we find a formula for $C(n, n-3)$ consisting of multiple terms, the majority of which involve sums. This proof is similar to the initial argument of Theorem~\ref{2diag} concerning $C(n, n-2)$, but as should be expected is more complicated and requires six cases. For this reason we defer the proof to Appendix \ref{sec:app proof}. Utilizing this unwieldy formula we are able to prove that $C(n, n-3)$ satisfies a surprising recursive relation for large enough values of $n$. From this recursive relation we determine a generating function for $C(n, n-3)$ which leads to a much more compact closed form formula for the same values.

\begin{lemma}\label{3diagstr}
We have that $$C(n, n-3)=4\sum_{m=4}^{n-2}\sum_{i=1}^{n-m-1}2^{n-m-2}+8\sum_{m=4}^{n-1}2^{n-m-1}+4\sum_{m=4}^{n-3}\sum_{i=1}^{n-m-2}\sum_{j=1}^{n-m-i-1}2^{n-m-3}(m-3)+$$ $$+12\sum_{m=4}^{n-2}\sum_{i=1}^{n-m-1}2^{n-m-2}(m-3)+12\sum_{m=4}^{n-1}2^{n-m-1}(m-3)+4(n-3)+2\sum_{i=1}^{n-5}2^{n-6}+2^{n-3}+$$ $$+2\sum_{m=3}^{n-2}\sum_{i=1}^{n-m-1}2^{n-m-2}(m-2)+4\sum_{m=3}^{n-1}2^{n-m-1}(m-2)+2(n-2)+4\sum_{i=1}^{n-4}2^{n-5}+2^{n-1}+4$$
\end{lemma}
\begin{proof}
Deferred to Appendix \ref{sec:app proof}.
\end{proof}

\noindent
In the proof of Theorem~\ref{3recur} below we make use of the following identities, which can be verified inductively:
\begin{equation}\label{eq:k2^k}
\sum_{k=1}^n(n-k)2^k=2^{n+1}-2n-2,
\end{equation}
\begin{equation}\label{eq:k^22^k}
\sum_{k=1}^nk(n-k)2^{n-k-1}=4-3\cdot2^n+n\cdot2^n.
\end{equation}

\begin{theorem}\label{3recur}
$50C(n, n-1)+8C(n+1, n-1)+2C(n+4, n+1)=C(n+5, n+2)$ for $n\ge 1$.
\end{theorem}
\begin{proof}
Using Lemma~\ref{3diagstr}, subtract $2C(n+4, n+1)$ from $C(n+5, n+2)$ to yield $$28n+24+12\sum_{m=4}^{n+2}2^{n-m+3}(m-3)+5\cdot 2^n+4\sum_{m=4}^{n+2}2^{n-m+3}+4\sum_{m=4}^{n+1}\sum_{i=1}^{n-m+2}2^{n-m+2}(m-3)+$$
$$+4\sum_{m=4}^{n+1}2^{n-m+2}(m-3)+2\sum_{m=3}^{n+2}2^{n-m+3}(m-2).$$
Using \eqref{eq:k2^k} and \eqref{eq:k^22^k} above along with the formula for the partial sum of a geometric series, this becomes 
$$28n+24-24+24\cdot 2^n-24n+5\cdot 2^n-8+4\cdot 2^n+16-4\cdot 2^{n+2}+8n+4n2^n+4\cdot 2^n+4\cdot2^n-8n-8+4\cdot2^{n+1}-4n$$
$$=50\cdot 2^{n-1}+8n2^{n-1}$$
$$=50C(n, n-1)+8C(n+1, n-1).$$
\end{proof}

\noindent
Using the recursive relation of Theorem \ref{3recur} and the formulas for $C(n, n-1)$ and $C(n+1, n-1)$, we get the following corollary. 

\begin{corollary}
For $n\ge 3,$
\begin{eqnarray}\label{codim 3}
\sum_{n>0}C(n, n-3)x^n=\frac{6x^3-10x^4-4x^5+10x^6-4x^7}{(1-2x)^3}.
\end{eqnarray}
\end{corollary}

\noindent
Now, using Theorem 4.1.1. in \textbf{\cite{S1}}, we can extract the following interesting theorem.

\begin{theorem}
For $n\ge 3$, 
\[C(n, n-3) =  \begin{cases} 
      (7n-15)2^{n-3} & 3\le n\le 5 \\
      (\frac{1}{2}n^2+\frac{11}{4}n-\frac{25}{4})2^{n-3} & n\ge5.
   \end{cases}
\]
\end{theorem}

\section{More Formulas}

In counting ordered pairs of compositions corresponding to seaweeds with a certain index, it is worth noting that, as with Duflo, we are not enumerating
the number of such seaweeds up to conjugation.  However, the parabolic case 
is quite different from the biparabolic case.  In the parabolic case, the  conjugacy classes of parabolic subalgebras of $\mf{sl}(n)$ are in one-to-one correspondence with compositions of $n$.  Leveraging the formulas in  Section 4.1, we can enumerate the number of conjugate parabolics precisely.

\subsection{Linear GCD Formulas for the Index}

The discrete combinatorial formula of Dergachev and Kirillov given in Theorem \ref{thm:dk}, while elegant, is difficult to apply in practice. However, in certain cases, the following index formulas allow us to ascertain the index directly from the block sizes of the flags that define the seaweed. 

The following formulas were developed in a series of articles \textbf{\cite{Coll1,Coll2, Coll3, Collar}}. The first formula for maximal parabolics (\ref{2parts}) was known in its essential form to Elashvili as early as 1990 (see (\textbf{\cite{Elash}}), but, together with the introduction of the latter formula (\ref{3parts}), was reestablished using different methods by Coll et al in 2015 (see \textbf{\cite{Coll2}}, cf., \textbf{\cite{dk}}.) 

\begin{theorem}[Coll et al, \textbf{\cite{Coll2}}]\label{2parts}
A seaweed of type $\dfrac{a|b}{n}$ has index $\gcd (a,b)-1.$
\end{theorem}

\begin{theorem}[Coll et al,\textbf{ \cite{Coll2}}]\label{3parts}
A seaweed of type $\dfrac{a|b|c}{n}$, or type $\dfrac{a|b}{c|n-c}$, has index $\gcd (a+b,b+c)-1$.
\end{theorem}

In response to a conjecture of the first author and Magnant in \textbf{\cite{Collar}}, the following recent result establishes that the formulas in Theorems \ref{2parts} and \ref{3parts} are the only nontrivial linear ones that are available in the parabolic case.  

\begin{theorem}[Theorem 5.3,\textbf{ \cite{Kar}}]  If $m\geq 4$ and   
$\mf{p}$ is a seaweed of type $\dfrac{a_1|a_2|\cdots|a_m}{n}$, then there do not exist homogeneous polynomials $f_1,f_2 \in \Z[x_1,...,x_m]$ of arbitrary degree, such that the index of $\mf{p}$ is given by 
$\gcd (f_1(a_1,...,a_m),f_2(a_1,...,a_m))$.

\end{theorem}
\subsection{$C_{2,1}(n,k)$}\label{subsec:maximal}

In this section, we are concerned with computing the number of conjugacy classes of maximal parabolic subalgebras of $\mf{sl}(n)$ with index $k$, which we denote $C_{2,1}(n,k)$. The conjugacy classes of such subalgebras are in one-to-one correspondence with seaweeds of the form $ \frac{a|b}{n}$. The subscripts in our notation for $C_{2,1}(n,k)$ are suggestive of the number of parts in the top and bottom compositions of the conjugacy class representatives that we will be counting.

The result of this section is an immediate consequence of Theorem \ref{2parts}.

\begin{theorem}\label{com21} If $n=(k+1)t$ for some integer $t$,
then $C_{2,1}(n, k)=\varphi(t)$, where $\varphi$ is Euler's totient function.
Otherwise, $C_{2,1}(n, k)=0$.
\end{theorem}

\begin{proof}
For a given positive integer $n$, there are $n-1$ seaweeds of the form $\frac{a|b}{n}$. The index of such a seaweed is equal to $\gcd(a,b)-1$ (see \cite{dk}). Since $\gcd(a,b)=\gcd(a,a+b)$, our goal is to find the cardinality of the following set
\[\{a\in\mathbb{Z} \mid 1\leq a\leq n-1 \text{ and }\gcd(a,n)=k+1\}.\]
Thus, $n=(k+1)t$ for some integer $t$, otherwise $C_{2,1}(n, k)=0$. Since $k+1$ must also divide $a$, we can express each element of this set as $a=(k+1)s$ for some integer $s$. To ensure that $\gcd(a,n)=k+1$, $s$ must be relatively prime to $t$, and furthermore $s$ must be less than $t$ to ensure that $a\leq n-1$. It follows that $C_{2,1}(n, k)=\varphi(t)$.
\end{proof}
\bigskip
\noindent
\textbf{Example: } Consider the following Table which encodes values of $C_{2,1}(t+tk,k)$, for various values of $t$.  More specifically, $C_{2,1}(2+2k, k)$ are colored blue, $C_{2,1}(3+3k, k)$ are colored red, $C_{2,1}(4+4k, k)$ are colored gray, and $C_{2,1}(5+5k, k)$ are colored green.

\begin{table}[H]
\centering
\begin{tabular}{|c|c|c|c|c|c|c|c|c|c|c|c|c|}
\hline
\textbf{n\textbackslash k} & \textbf{0}  & \textbf{1} & \textbf{2} & \textbf{3} & \textbf{4} & \textbf{5} & \textbf{6} & \textbf{7} & \textbf{8} & \textbf{9} & \textbf{10} & \textbf{11} \\ \hline
\textbf{2}          & \cellcolor{blue!25}1  & 0 & 0 & 0 & 0 & 0 & 0 & 0 & 0 & 0 & 0  & 0  \\ \hline
\textbf{3}          & \cellcolor{red!25}2  & 0 & 0 & 0 & 0 & 0 & 0 & 0 & 0 & 0 & 0  & 0  \\ \hline
\textbf{4}          & \cellcolor{yellow!25}2  & \cellcolor{blue!25}1 & 0 & 0 & 0 & 0 & 0 & 0 & 0 & 0 & 0  & 0  \\ \hline
\textbf{5}          & \cellcolor{gray!25}4  & 0 & 0 & 0 & 0 & 0 & 0 & 0 & 0 & 0 & 0  & 0  \\ \hline
\textbf{6}          & \cellcolor{green!25}2  & \cellcolor{red!25}2 & \cellcolor{blue!25}1 & 0 & 0 & 0 & 0 & 0 & 0 & 0 & 0  & 0  \\ \hline
\textbf{7}          & 6  & 0 & 0 & 0 & 0 & 0 & 0 & 0 & 0 & 0 & 0  & 0  \\ \hline
\textbf{8}          & 4  & \cellcolor{yellow!25}2 & 0 & \cellcolor{blue!25}1 & 0 & 0 & 0 & 0 & 0 & 0 & 0  & 0  \\ \hline
\textbf{9}          & 6  & 0 & \cellcolor{red!25}2 & 0 & 0 & 0 & 0 & 0 & 0 & 0 & 0  & 0  \\ \hline
\textbf{10}         & 4  & \cellcolor{gray!25}4 & 0 & 0 & \cellcolor{blue!25}1 & 0 & 0 & 0 & 0 & 0 & 0  & 0  \\ \hline
\textbf{11}         & 10 & 0 & 0 & 0 & 0 & 0 & 0 & 0 & 0 & 0 & 0  & 0  \\ \hline
\textbf{12}         & 4  & \cellcolor{green!25}2 & \cellcolor{yellow!25}2 & \cellcolor{red!25}2 & 0 & \cellcolor{blue!25}1 & 0 & 0 & 0 & 0 & 0  & 0  \\ \hline
\end{tabular}
\caption{$C_{2,1}(n, k)$}
\label{tab:maxpara}
\end{table}

\bigskip
\noindent
\textit{Remark:}  Although we have presented the central theorem of this section first, and the table examples later, it was the empirical data of Table 2 that provided clues to special cases of the general theorem from which they now follow as a consequence.

\subsection{$C_{2,2}(n,k)$}

In this section a formula is determined which characterizes Table~\ref{22table} below which enumerates $C_{2,2}(n, k)$.

\begin{table}[H]
\centering
\begin{tabular}{|c|c|c|c|c|c|c|c|c|c|c|c|}
\hline
\textbf{Dim\textbackslash Index} & \textbf{0}  & \textbf{1}  & \textbf{2}  & \textbf{3} & \textbf{4} & \textbf{5} & \textbf{6} & \textbf{7} & \textbf{8} & \textbf{9} & \textbf{10} \\ \hline
\textbf{2}        & \cellcolor{red!25}0  & \cellcolor{blue!25}1  & 0  & 0 & 0 & 0 & 0 & 0 & 0 & 0 & 0  \\ \hline
\textbf{3}        & \cellcolor{yellow!25}2  & 0  & \cellcolor{blue!25}2  & 0 & 0 & 0 & 0 & 0 & 0 & 0 & 0  \\ \hline
\textbf{4}        & \cellcolor{green!25}4  & \cellcolor{red!25}2  & 0  & \cellcolor{blue!25}3 & 0 & 0 & 0 & 0 & 0 & 0 & 0  \\ \hline
\textbf{5}        & \cellcolor{gray!25}12 & 0  & 0  & 0 & \cellcolor{blue!25}4 & 0 & 0 & 0 & 0 & 0 & 0  \\ \hline
\textbf{6}        & 8  & \cellcolor{yellow!25}8  & \cellcolor{red!25}4  & 0 & 0 & \cellcolor{blue!25}5 & 0 & 0 & 0 & 0 & 0  \\ \hline
\textbf{7}       & 30  & 0  & 0  & 0 & 0 & 0 & \cellcolor{blue!25}6 & 0 & 0 & 0 & 0  \\ \hline
\textbf{8}        & 24 & \cellcolor{green!25}12 & 0  & \cellcolor{red!25}6 & 0 & 0 & 0 & \cellcolor{blue!25}7 & 0 & 0 & 0  \\ \hline
\textbf{9}        & 42 & 0  & \cellcolor{yellow!25}14 & 0 & 0 & 0 & 0 & 0 & \cellcolor{blue!25}8 & 0 & 0  \\ \hline
\textbf{10}       & 32 & \cellcolor{gray!25}32 & 0  & 0 & \cellcolor{red!25}8 & 0 & 0 & 0 & 0 & \cellcolor{blue!25}9 & 0  \\ \hline
\textbf{11}       & 90 & 0  & 0  & 0 & 0 & 0 & 0 & 0 & 0 & 0 & \cellcolor{blue!25}10 \\ \hline
\end{tabular}
\caption{$C_{2,2}(n, k)$}~\label{22table}
\end{table}

The following theorem completely characterizes the non-zero entries in Table \ref{22table} and is a consequence of Theorem \ref{3parts}.

\begin{theorem}~\label{22char}
$C_{2,2}(t+tk, k)=(t+tk-2)\varphi(t)$ for integers $t>1$ and $k\ge 0$, where $\varphi$ is Euler's Totient function. When $t=1$, $C_{2,2}(t+tk, k)=t+tk-1$.
\end{theorem}
\begin{proof}
Begin by noting that the $t=1$ case follows by reasoning similar to that given in Section 3.1. Thus, we assume $t>1$.

To fix notation, recall that we are considering seaweeds of type $\frac{a|b}{c|d}$ with $a+b=c+d=n$. Since it is assumed that $t>1$, it is the case that either $a+d<n$ or $b+c<n$. To see this, note that we must have $a+b+c+d=2n$ so that at worst $a+d=b+c=n$; but if this is the case, then $a=c$ and $b=d$ which corresponds to the case $t=1$. Thus, throughout we will assume $b+c<n$.

Next, the result above will be used to find a formula for $C_{2,2}(t+tk, k)$ at $k=0$, i.e., for $C_{2,2}(t, 0)$. Using Theorem~\ref{3parts} one finds that $$C_{2,2}(t, 0)=2\cdot|\{(b, c)|b,c<t, \gcd(t, b+c)=1\}|,$$ where the factor of 2 allows for the case that $b+c>t$ (i.e., $a+d<t$). By definition $$|\{x|x<t, \gcd(t, x)=1\}|=\varphi(x)$$ and a little thought shows that each $s\in\{x|x<t, \gcd(t, x)=1\}$ corresponds to $s-1$ pairs $$(b_1, c_1)\in\{(b, c)|b,c<t, \gcd(t, b+c)=1\}.$$ Thus, it must be the case that $C_{2,2}(t, 0)=2\sum_{i=1}^{\varphi(t)}(s_i-1)$, where $\{s_1,...,s_{\varphi(t)}\}=\{x|x<t, \gcd(t, x)=1\}.$ Using a classic result on the sum of positive integers less than and relatively prime to a positive integer $t$: $$C_{2,2}(t, 0)=2\left(\frac{t}{2}\varphi(t)-\varphi(t)\right)=t\varphi(t)-2\varphi(t).$$

Finally, it is claimed that $C_{2,2}(t+tk, k)-C_{2,2}(tk, k-1)=t\phi(t)$. Note, that as above 

$$C_{2,2}(tk, k-1)=2\cdot|\{(b, c)|b,c<tk, \gcd(tk, b+c)=k\}|.$$ Applying similar reasoning to the case $k=0$, the value $|\{x|x<tk, \gcd(tk, x)=k\}|$ is analyzed first. Basic properties of the $\gcd$ gives $p\in\{x<t|\gcd(t, x)=1\}$ if and only if $pk\in\{x<tk| \gcd(tk, x)=k\}.$ It follows that $$\varphi(t)=|\{x<t| \gcd(t, x)=1\}|=|\{x<tk| \gcd(tk, x)=k\}|.$$ Thus, similar to the case $k=0$ and now assuming 
that $\{s_1,...,s_{\varphi(t)}\}=\{x|x<t, \gcd(t, x)=1\}$ we get that $$C_{2,2}(tk, k-1)=2\sum_{i=1}^{\varphi(t)}(s_ik-1)=2(\frac{tk}{2}\varphi(t)-\varphi(t))=tk\varphi(t)-2\phi(t).$$ Therefore, $$C_{2,2}(t+tk, k)-C_{2,2}(tk, k-1)=t(k+1)\varphi(t)-2\varphi(t)-(tk\varphi(t)-2\phi(t))=t\phi(t).$$ Hence, for $t>1$, it follows that $C_{2,2}(t+tk, k)=(t+tk-2)\varphi(t)$.
\end{proof}

\section{Afterword}
\noindent
This initial investigation makes use of recent Lie algebraic technology (meanders, homotopy types, and gcd index formulas) to enumerate composition types associated with seaweed subalgebras in Type A. By doing so, it provides the framework for similar investigations in the other classical families, which likewise require analogous technologies -- recently advanced by several investigative groups as follows.

In \textbf{\cite{Panyushev1}}, Panyushev extended the Lie theoretic definition of seaweed subalgebras to the reductive algebras. If $\mf{p}$ and $\mf{p'}$ are parabolic subalgebras of a reductive Lie algebra $\mf{g}$ such that $\mf{p}+\mf{p'}=\mf{g}$, then $\mf{p}\cap\mf{p'}$ is called a \textit{seaweed subalgebra o}f $\mf{g}$ or simply $seaweed$ when $\mathfrak{g}$ is understood. For this reason, Joseph  has elsewhere 
 \textbf{\cite{Joseph}} called seaweed algebras, \textit{biparabolic}.  One can show that Type-C and Type-B seaweeds, in their standard representation, can be parametrized by a pair of partial compositions of $n$.  Indeed, in \textbf{\cite{CHM}}, Coll et al have topically extended the Type-A work of Dergachev and A. Kirillov to the Type-B and Type-C cases, providing analogous definitions of meanders and index formulas; see also \textbf{\cite{Panyushev2}}, where Type-C meanders were independently developed -- absent the index formulas based on the compositions which define the seaweed. The homotopy types for Types B and C have also been classified by the second and fourth authors who have additionally shown in unpublished work that, at least in the Type C case, the homotopy type is a conjugation invariant. Preliminary results in Types B and C suggest that the generating functions for $C(n, n-k)$, i.e., $\sum_{n>0}C(n, n-k)$, are of the form $\frac{f(x)}{(1-2x)^{k+1}}$ much like that of Type-A, which seem to be of the form $\frac{f(x)}{(1-2x)^k}$ where $f(x)$ is a polynomial with integer coefficients and has the same degree in both cases. In Types A, B, and, C, precise enumerative formulas for the number of maximal parabolic seaweeds are made possible by linear gcd index formulas.   The complete compliment of such formulas has been developed in \textbf{\cite{Coll2}} for Type-A and in \textbf{\cite{CHM}}, for Types B and C.
 
Most recently in \textbf{\cite{Panyushev3}}, Panychev and Yakimova have developed Type-D meanders and Cameron et al \textbf{\cite{Cameron}} have completed the classification of linear greatest common divisor formulas for the classical families by providing index formulas based on the defining compositions associated with a Type-D seaweed.  The Type-D case is complicated by the bifurcation point in the Dynkin diagram associated with Type-D Lie algebras, which amongst other things, allows for certain biparablics to not have the distinctive seaweed shape in their standard representations.

Follow-up work will provide analogues of the main theorems in this paper to the other classical types.



\begin{appendices}
\section{The Signature of a Meander}\label{sec:app signature}

The following lemma is a graph-theoretic reductive rendering of the well-known inductive formula of Panyushev (\textbf{\cite{Panyushev1}}, Theorem 4.2).

\begin{lemma}[Winding-down] Given a meander $M$ of type $\dfrac{a_1|a_2|...|a_m}{b_1|b_2|...|b_t}$, create a meander $M'$
by exactly one of the following moves. For all moves except the 
Component Elimination move, $M$ and $M'$ have the same homotopy type. 
\begin{enumerate}
\item {\bf Flip $(F)$:} If $a_1<b_1$, then 
$M'$ has type 
$\displaystyle\frac{b_1|b_2|...|b_t}{a_1|a_2|...|a_m}$.
    
\item {\bf Component Elimination $(C(c))$:} 
If $a_1=b_1=c$, then $M'$ has type 
$\displaystyle\frac{a_2|a_3|...|a_m}{b_2|b_3|...|b_t}.$

\item {\bf Rotation Contraction $(R)$:} If $b_1<a_1<2b_1$, then 
$M'$ has type 
$\displaystyle \frac{b_1|a_2|a_3|...|a_m}{(2b_1-a_1)|b_2|...|b_t}$.

\item {\bf Block Elimination $(B)$:} If $a_1=2b_1$, then 
$M'$ has type 
$\displaystyle\frac{b_1|a_2|..|a_m}{b_2|b_3|...|b_t}$.

\item {\bf Pure Contraction $(P)$:} If $a_1>2b_1$, then 
$M'$ has type 
$\displaystyle
\frac{(a_1-2b_1)|b_1|a_2|a_3|...|a_m}{b_2|b_3|...|b_t}$.

\end{enumerate}
\end{lemma}

Given a meander, there exists a unique sequence of moves (elements of $\{F, C(c), R, B, P \})$ which reduce the meander down to its plane homotopy type. Such a list is called the \textit{signature} of the meander.

\begin{ex}
Consider the meander for $\frac{15}{2|5|1|5|2}$ which can be wound down to yield a planar graph with homotopy type $H(1,5,2)$, cf., Figure 2.

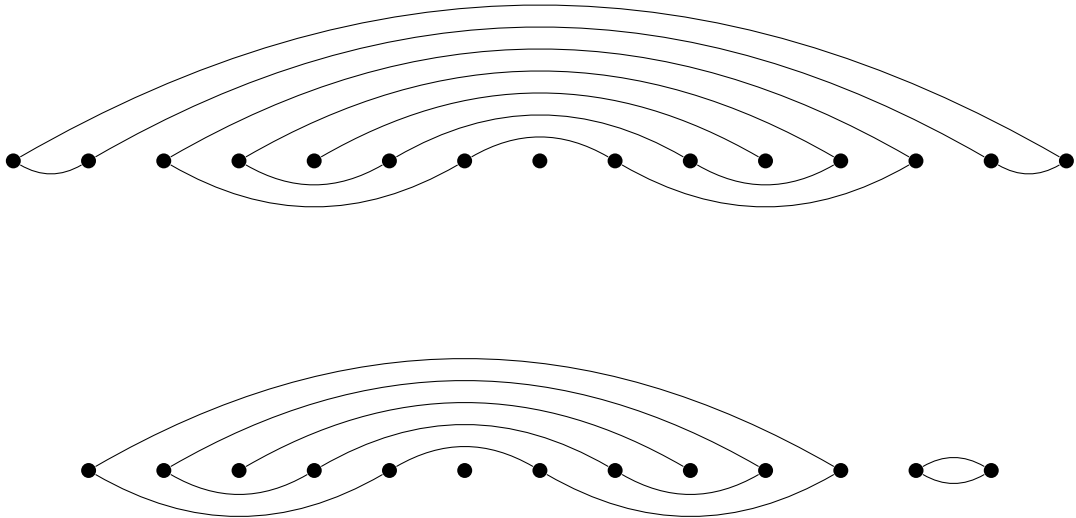
\begin{figure}[H]
$$\begin{tikzpicture}
	\def\Node{\node [circle,  fill, inner sep=2pt]}
	\Node (1) at (0, 0) {};
    \Node (2) at (1, 0) {};
    \Node (3) at (2, 0) {};
    \Node (4) at (3, 0) {};
    \Node (5) at (4, 0) {};
    \Node (6) at (5, 0) {};
    \Node (7) at (6, 0) {};
    \Node (8) at (7, 0) {};
    \Node (9) at (8, 0) {};
    \Node (10) at (9, 0) {};
    \Node (11) at (10, 0) {};
    \Node (12) at (11, 0) {};
    \Node (13) at (12, 0) {};
    \Node (14) at (13, 0) {};
    \Node (15) at (14, 0) {};
    \draw (1) to[bend left] (15);
    \draw (2) to[bend left] (14);
    \draw (3) to[bend left] (13);
    \draw (4) to[bend left] (12);
    \draw (5) to[bend left] (11);
    \draw (6) to[bend left] (10);
    \draw (7) to[bend left] (9);
    \draw (1) to[bend right] (2);
    \draw (3) to[bend right] (7);
    \draw (4) to[bend right] (6);
    \draw (9) to[bend right] (13);
    \draw (10) to[bend right] (12);
    \draw (14) to[bend right] (15);
\end{tikzpicture}$$
\end{figure}
\begin{figure}[H]
$$\begin{tikzpicture}
	\def\Node{\node [circle,  fill, inner sep=2pt]}
    \Node (3) at (2, 0) {};
    \Node (4) at (3, 0) {};
    \Node (5) at (4, 0) {};
    \Node (6) at (5, 0) {};
    \Node (7) at (6, 0) {};
    \Node (8) at (7, 0) {};
    \Node (9) at (8, 0) {};
    \Node (10) at (9, 0) {};
    \Node (11) at (10, 0) {};
    \Node (12) at (11, 0) {};
    \Node (13) at (12, 0) {};
    \Node (14) at (13, 0) {};
    \Node (15) at (14, 0) {};
    \draw (14) to[bend left] (15);
    \draw (3) to[bend left] (13);
    \draw (4) to[bend left] (12);
    \draw (5) to[bend left] (11);
    \draw (6) to[bend left] (10);
    \draw (7) to[bend left] (9);
    \draw (3) to[bend right] (7);
    \draw (4) to[bend right] (6);
    \draw (9) to[bend right] (13);
    \draw (10) to[bend right] (12);
    \draw (14) to[bend right] (15);
\end{tikzpicture}$$
\end{figure}
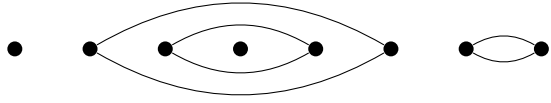
\begin{figure}[H]
$$\begin{tikzpicture}
	\def\Node{\node [circle,  fill, inner sep=2pt]}
    \Node (8) at (7, 0) {};
    \Node (9) at (8, 0) {};
    \Node (10) at (9, 0) {};
    \Node (11) at (10, 0) {};
    \Node (12) at (11, 0) {};
    \Node (13) at (12, 0) {};
    \Node (14) at (13, 0) {};
    \Node (15) at (14, 0) {};
    \draw (14) to[bend left] (15);
    \draw (9) to[bend left] (13);
    \draw (10) to[bend left] (12);
    \draw (9) to[bend right] (13);
    \draw (10) to[bend right] (12);
    \draw (14) to[bend right] (15);
\end{tikzpicture}$$
\caption{ $\frac{15}{2|5|1|5|2}$ with has signature $PPC(1)C(5)C(2)$}
\end{figure}
\end{ex}

\section{Proof of Lemma \ref{3diagstr}}\label{sec:app proof}

\begin{proof}

A meander with $n$ vertices has index $n-3$ if and only if it has one of the following forms:
\begin{itemize}
\item One cycle containing exactly four vertices, all other cycles contain exactly two vertices, and all paths contain only one vertex. We count the number of such meanders in Case \ref{4 vertex cycle} below.
\item All cycles contain exactly two vertices, two paths that each contain exactly two vertices, and all other paths contain only one vertex. These meanders are treated in Cases \ref{2 paths part 1}, 
\ref{2 paths part 2}, and \ref{2 paths part 3} below.
\item All cycles contain exactly two vertices, one path that contains exactly three vertices, and all other paths contain only one vertex. These meanders are treated in Cases \ref{3 vertex path part 1} 
and \ref{3 vertex path part 2} below.
\end{itemize}

\begin{case}\label{4 vertex cycle}
\end{case}
The meander has a cycle that contains four vertices. Such a meander must have form
\[\frac{\bm{a}|2|m-4|2|\bm{c}}{\bm{a}|m|\bm{c}} \hspace{.5cm}\text{or}\hspace{.5cm}
\frac{\bm{a}|m|\bm{c}}{\bm{a}|2|m-4|2|\bm{c}},a\]
where $m$ is an integer such that $4\leq m\leq n$, and $\bm{a}$ and $\bm{c}$ are possibly empty compositions.
If both $\bm{a}$ and $\bm{c}$ are nonempty, then there are
\begin{equation}\label{eq:first}
2\sum_{m=4}^{n-2}\sum_{i=1}^{n-m-1}2^{n-m-2},
\end{equation}
meanders of this form. In the above equation $i$ represents the sum of the parts of $\bm{a}$.

If $\bm{a}$ or $\bm{c}$ (or both) are empty, we get a contribution (to $C(n, n-3)$) of
\begin{equation}
2+4\sum_{m=4}^{n-1}2^{n-m-1},
\end{equation}
such meanders.

\begin{case}\label{2 paths part 1}
\end{case}
The meander has two paths containing two vertices such that one block contains
all four of these vertices. Such a meander must have form
\[\frac{\bm{a}|1|1|m-4|1|1|\bm{c}}{\bm{a}|m|\bm{c}} \hspace{.5cm}\text{or}\hspace{.5cm}
\frac{\bm{a}|m|\bm{c}}{\bm{a}|1|1|m-4|1|1|\bm{c}}\]
where $m$ is an integer such that $4\leq m\leq n$,
and $\bm{a}$ and $\bm{c}$ are possibly empty compositions.
The number of such meanders is identical to that of Case \ref{4 vertex cycle},
so we get contributions of
\begin{equation}
2\sum_{m=4}^{n-2}\sum_{i=1}^{n-m-1}2^{n-m-2},
\end{equation}
and
\begin{equation}
2+4\sum_{m=4}^{n-1}2^{n-m-1}.
\end{equation}


\begin{case}\label{2 paths part 2}
\end{case}
The meander has two paths containing two vertices, such that there is no block containing vertices from both of these paths. Such a meander must have the form
$$\dfrac{\bm{a}|\bm{b}|\bm{c}|\bm{d}|\bm{e}}{\bm{a}|\bm{b}'|\bm{c}|\bm{d}'|\bm{e}},$$
where
\[\dfrac{\bm{b}}{\bm{b}'}=\frac{k}{1|k-2|1}
\hspace{.5cm}\text{or}\hspace{.5cm}\dfrac{\bm{b}}{\bm{b}'}=\frac{1|k-2|1}{k},
\hspace{.5cm}\text{ and }\hspace{.5cm}
\dfrac{\bm{d}}{\bm{d}'}=\frac{l}{1|l-2|1}
\hspace{.5cm}\text{or}\hspace{.5cm}\dfrac{\bm{d}}{\bm{d}'}=\frac{1|l-2|1}{l},\] 
for some integers $k\geq 2$ and $l\geq 2$ such that $k+l\leq n$,
and $\bm{a}, \bm{c}$, and $\bm{e}$ are possibly empty compositions.
If all of $\bm{a}, \bm{c}$, and $\bm{e}$ are nonempty, we get a contribution of 
\begin{equation}
4\sum_{m=4}^{n-3}~\sum_{i=1}^{n-m-2}~\sum_{j=1}^{n-m-i-1}2^{n-m-3}(m-3),
\end{equation}
where $m=l+k$, $i$ is the sum of the parts of $\bm{a}$, 
and $j$ is the sum of the parts of $\bm{c}$.

If exactly one of $\bm{a}, \bm{c}$, or $\bm{e}$ is empty, we get a contribution of 
\begin{equation}
12\sum_{m=4}^{n-2}\sum_{i=1}^{n-m-1}2^{n-m-2}(m-3),
\end{equation}
where $m=l+k$ and $i$ is the sum of the parts of $\bm{a}$.

If exactly two of $\bm{a}, \bm{c}$, or $\bm{e}$ is empty, we get a contribution of 
\begin{equation}
12\sum_{m=4}^{n-1}2^{n-m-1}(m-3),
\end{equation}
where $m=l+k$.

Finally, if all of $\bm{a}, \bm{c}$, and $\bm{e}$ are empty, we get a contribution of 
\begin{equation}
4(n-3).
\end{equation}

\begin{case}\label{2 paths part 3}
\end{case}
The meander has two paths containing two vertices, such that there is a block containing at least one vertex, but no more than three vertices from both of these paths. Such a meander must have the form 
\[\dfrac{\bm{a}|3|1|\bm{b}}{\bm{a}|1|3|\bm{b}}\hspace{.5cm}\text{or}\hspace{.5cm}
\dfrac{\bm{a}|1|3|\bm{b}}{\bm{a}|3|1|\bm{b}},\]
where $\bm{a}$ and $\bm{b}$ are possibly empty compositions.
If both $\bm{a}$ and $\bm{b}$ are nonempty, we get a contribution of
\begin{equation}
2\sum_{i=1}^{n-5}2^{n-6},
\end{equation}
where $i$ is the sum of the parts of $\bm{a}$.

If either $\bm{a}$ or $\bm{b}$ is empty (or both if $n=4$), then we get a contribution of 
\begin{equation}
4\left(2^{n-5}\right)=2^{n-3}.
\end{equation}

\begin{case}\label{3 vertex path part 1}
\end{case}
The meander has a path containing three vertices, such that there is no block that 
contains both of the endpoints of this path. Such a meander must have the form
\[\frac{\bm{a}|1|l-2|k|\bm{c}}{\bm{a}|l|k-2|1|\bm{c}}\hspace{.5cm}\text{or}\hspace{.5cm}
\frac{\bm{a}|l|k-2|1|\bm{c}}{\bm{a}|1|l-2|k|\bm{c}},\]
for some integers $k\geq 2$ and $l\geq 2$ such that $k+l-1\leq n$, 
and $\bm{a}$ and $\bm{c}$ are possibly empty compositions.
If both of $\bm{a}$ and $\bm{c}$ are nonempty, we get a contribution of
\begin{equation}
2\sum_{m=3}^{n-2}\sum_{i=1}^{n-m-1}2^{n-m-2}(m-2),
\end{equation}
where $m=l+k-1$ and $i$ is the sum of the parts of $\bm{a}$.

If exactly one of $\bm{a}$ or $\bm{c}$ is empty, we get a contribution of
\begin{equation}
4\sum_{m=3}^{n-1}2^{n-m-1}(m-2),
\end{equation}
where $m=l+k-1$.

If both of $\bm{a}$ and $\bm{c}$ are empty, we get a contribution of 
\begin{equation}
2(n-2).
\end{equation}

\begin{case}\label{3 vertex path part 2}
\end{case}
The meander has a path containing three vertices, such that there is a block that
contains both of the endpoints of this path. Such a meander must have the form
\[\frac{\bm{a}|1|2|\bm{b}}{\bm{a}|3|\bm{b}}\hspace{.5cm}\text{or}\hspace{.5cm}
\frac{\bm{a}|2|1|\bm{b}}{\bm{a}|3|\bm{b}}\hspace{.5cm}\text{or}\hspace{.5cm}
\frac{\bm{a}|3|\bm{b}}{\bm{a}|2|1|\bm{b}}\hspace{.5cm}\text{or}\hspace{.5cm}
\frac{\bm{a}|3|\bm{b}}{\bm{a}|1|2|\bm{b}},\]
where $\bm{a}$ and $\bm{b}$ are (possibly empty) compositions.
If both $\bm{a}$ and $\bm{b}$ are nonempty, we get a contribution of
\begin{equation}
4\sum_{i=1}^{n-4}2^{n-5},
\end{equation}
where $i$ is the sum of the parts of $\bm{a}$.

If either $\bm{a}$ or $\bm{b}$ is empty (or both are empty when $n=3$), then we get a contribution of
\begin{equation}\label{eq:last}
8\left(2^{n-4}\right)=2^{n-1}.
\end{equation}

\noindent
Therefore, $C(n,n-3)$ is equal to the sum of the expressions 
in \eqref{eq:first}--\eqref{eq:last}.

\end{proof}

\end{appendices}

\bibliographystyle{abbrv}

\bibliography{bibliography.bib}

\end{document}